\documentclass[11pt,a4paper]{article}

\usepackage{amsmath} 
\usepackage{mathtools} 
\usepackage{amssymb} 
\usepackage{stmaryrd} 
\usepackage{mathrsfs} 
\usepackage{amsthm}
\usepackage{microtype}

\usepackage[margin=3cm]{geometry}

\usepackage{todonotes}
\usepackage{latexsym}
\usepackage{xspace}

\usepackage{graphicx} 
\usepackage{float} 
\usepackage{caption} 
\usepackage{subcaption} 
\usepackage{enumitem}

\usepackage{siunitx} 
\usepackage{tikz}
\usepackage{tikz-cd}
\usepackage{blindtext} 
\usepackage{comment}
\usepackage{hyperref}
\hypersetup{hidelinks}
\usepackage{amsthm}
\usepackage{csquotes}

\usepackage{titlesec}

\usepackage{xcolor}
\IfFileExists{ulem.sty}{\usepackage[normalem]{ulem}}{} 

\theoremstyle{plain}
\newtheorem{theorem}{Theorem}[section]
\newtheorem{proposition}[theorem]{Proposition}
\newtheorem{lemma}[theorem]{Lemma}
\newtheorem{corollary}[theorem]{Corollary}

\theoremstyle{definition}

\newtheorem{remark}[theorem]{Remark}
\newtheorem{example}[theorem]{Example}

\newcommand{\vect}[1]{\text{Vect}}

\newcommand{\R}{\mathbb{R}}

\newcommand{\Q}{\mathbb{Q}}

\newcommand{\N}{\mathbb{N}}
\newcommand{\Z}{\mathbb{Z}}

\newcommand{\df}[1]{\,\mathrm{d}#1}

\newcommand{\norm}[1]{\|#1\|}

\newcommand{\scalp}[2]{\langle #1 \, , #2\rangle}

\newcommand{\OO}{\mathcal{O}}

\newcommand{\MSC}[1]{\begingroup\small\noindent\textbf{2020 Mathematics Subject Classification.} #1\par\smallskip\endgroup}
\newcommand{\Keywords}[1]{\begingroup\small\noindent\textbf{Keywords.} #1\par\endgroup}

\title{\textcolor{black}{Generic non-vanishing of Dirichlet eigenfunction averages and absence of symmetries}}
\author{Vincent Boulard\footnote{CERMICS, CNRS, École nationale des ponts et chaussées, Institut Polytechnique de Paris, Marne-la-Vall\'ee, France (\texttt{vincent.boulard@enpc.fr}).}}
\date{}
\begin{document}

\maketitle

\begin{abstract}
We prove that, for a generic $C^m$-domain
$\Omega \subset \R^d$ (with $m\geq3$, $d \geq 2$) in the sense of Baire category for the Micheletti
topology, every Dirichlet-Laplacian eigenfunction has nonzero average. This gives a standalone spectral-geometric answer to a question raised
by Steinerberger and Venkatraman. The proof is based on a direct
shape-perturbation argument.
A key ingredient, which appears to be of independent interest, is a Baire-category theorem for shapes: domains with real-analytic boundary form a meager subset of the space of $C^m$ domains. It is this input that allows us to bypass, rather than resolve pointwise, the overdetermined elliptic configurations in which the first shape derivative of the mean degenerates.
We then derive two consequences. First, from a geometric viewpoint, a generic
domain has trivial Euclidean isometry group. Second, in control theory, a generic domain makes the
Dirichlet heat equation approximately controllable and rapidly stabilizable,
that is, stabilizable at any prescribed exponential decay rate, by means of a
single spatially homogeneous scalar control.
\end{abstract}

{
\MSC{35P05, 58J50, 93B05.}
}

\tableofcontents

\section{Introduction}
{
Let $\Omega \subset \R^d$ be a bounded domain with regular boundary, and let
$(\varphi_n)_{n \geq 1}$ be an $L^2$-orthonormal family of Dirichlet-Laplacian
eigenfunctions,
\[
-\Delta \varphi_n = \lambda_n \varphi_n \quad \text{in } \Omega,
\qquad \varphi_n|_{\partial\Omega}=0.
\]
This article is concerned with the spectral coefficients of the constant
function along the Dirichlet eigenbasis, namely
\[
 \forall n \geq 1, \quad \int_\Omega \varphi_n(x)\,\df x.
\]
Under Neumann boundary conditions the analogous question is trivial: the first
eigenfunction is constant and all the others have zero mean by orthogonality. In
the Dirichlet case the situation is very different. The first eigenfunction has
nonzero mean, but the remaining modes are not constrained by orthogonality to the
constants.

The vanishing or non-vanishing of these coefficients is a simple-looking
spectral question, but it reflects the geometry of the domain. Symmetric domains
show that many means can vanish. On the cube $[0,1]^d$, the eigenfunctions are
products $\prod_i \sin(a_i\pi x_i)$, and the integral vanishes as soon as one of
the integers $a_i$ is even. On the ball, separation of variables forces every
non-radial mode to have zero mean. More generally, symmetries produce invisible
modes by parity. In the absence of symmetries, Steinerberger and Venkatraman
suggested that one should expect every Dirichlet eigenfunction to have nonzero
average, and asked whether this holds generically
\cite{SteinerbergerVenkatraman2025}. They also proved that infinitely many
Dirichlet eigenfunctions have nonzero mean on every smooth domain, and obtained a
universal quantitative lower bound for their number among the first $n$ modes.

The purpose of the present paper is to isolate this non-vanishing phenomenon as a generic 
spectral-geometric property of domains and to derive consequences for the symmetries of 
domains and the controllability of the Dirichlet heat equation by a single spatially homogeneous scalar control. 
The restriction $d \geq 2$ is essential: in dimension one, every domain is an interval, 
and every even-indexed Dirichlet eigenfunction has zero mean.
\subsection{Main spectral result}

We work in the Micheletti framework. Let $D\subset\R^d$ be a bounded connected
open set of class $C^m$, and denote by $\OO^m(D)$ the space of domains obtained
from $D$ by $C^m$ diffeomorphisms asymptotic to the identity, endowed with the
Courant--Micheletti distance. This is a complete separable metric space
\cite{Micheletti1972,Henry2005}, hence a Baire space. We also denote by
$\mathcal S\subset\OO^m(D)$ the set of domains with simple Dirichlet spectrum,
which is residual by the classical generic simplicity theorems of Micheletti and
Uhlenbeck \cite{Micheletti1972,Uhlenbeck1976}.

For $\Omega\in\mathcal S$, each eigenfunction is determined up to sign, so the
condition $\int_\Omega \varphi_n\neq 0$ is well defined. We prove the following.{\footnote{A statement equivalent to Theorem~\ref{thm:main} is implicitly contained in the proof of a control-theoretic result of Chitour, Coron and Garavello \cite{ChitourCoronGaravello2006}. See Section~\ref{subsec:previous} for a detailed comparison of the statements as well as of the proofs.}}

\begin{theorem}\label{thm:main}
{
Let $D \subset \R^d$ be a bounded connected open set of class $C^m$
($m \geq 3$, $d \geq 2$). Then the set
\[
 \Bigl\{ \Omega \in \mathcal{S} : \forall n \geq 1,\ \int_\Omega \varphi_n(\Omega)  \neq 0 \Bigr\}
\]
is a dense $G_\delta$ set, and in particular a residual set, in $\OO^m(D)$.
}
\end{theorem}

This places the nonvanishing of the mean within a long line of generic spectral
properties of the Dirichlet Laplacian. Generic simplicity of the spectrum, under
perturbation of the domain or of the metric, goes back to
\cite{Micheletti1972,Uhlenbeck1976} and Albert
\cite{Albert1975, Albert1978}, who also established that eigenfunctions are
generically Morse functions. The codimension and analytic-path structure
underlying such statements was clarified by Teytel \cite{Teytel1999}. Closer to
the present setting, Privat and Sigalotti showed that the squares
$\varphi_n^2$ are generically linearly independent and that the spectrum is
generically non-resonant \cite{PrivatSigalotti2010}, Pereira and Pereira proved the generic non-vanishing of the cubic integrals $\int_\Omega \varphi_n^3$ \cite{PereiraPereira2002}, while Ortega and Zuazua
developed the same circle of ideas, combining Baire category with shape
differentiation, for the plate equation \cite{OrtegaZuazua2000, OrtegaZuazua2003}.

\subsection{Generic triviality of the isometry group}

For a bounded regular connected open subset $\Omega \subset \R^d$ we write
\begin{equation*}
    \mathrm{Iso}(\Omega) := \{ S \in \mathrm{Isom}(\R^d) \, : \, S(\Omega) = \Omega \}
\end{equation*}
for its group of Euclidean symmetries.
A domain with a symmetry is a degenerate object, and one expects a generic domain
to have none. For Riemannian metrics on a fixed compact manifold, the corresponding statement is classical: metrics without isometries form an open dense set, as shown by Ebin \cite{Ebin1970} (see Mounoud \cite{Mounoud2015} for the pseudo-Riemannian case). We are not aware of a counterpart for Euclidean domains, where the perturbation acts on the shape rather than on the metric tensor. The spectral route below yields it as a byproduct of Theorem~\ref{thm:main}, together with quantitative information on the possible symmetry groups.

The natural attempt to prove this runs through the spectrum. Any
$S \in \mathrm{Iso}(\Omega)$ preserves the Lebesgue measure and commutes with the
Dirichlet Laplacian, so when $\lambda_n$ is simple it acts on the corresponding
eigenfunction by a sign,
\begin{equation}\label{eq:sign-character-intro}
    \varphi_n \circ S = \varepsilon_n(S)\, \varphi_n, \qquad \varepsilon_n(S) \in \{\pm 1\}.
\end{equation}
The eigenfunctions form a Hilbert basis of $L^2(\Omega)$, so $S$ is entirely determined by the collection of signs
$(\varepsilon_n(S))_{n \geq 1}$. One is tempted to conclude that all these signs
equal $+1$, and therefore that $S = \mathrm{Id}$.

This conclusion is false, and the simplicity of the spectrum is not what is
missing. Simplicity does constrain $\mathrm{Iso}(\Omega)$, but only to the extent
of forcing it to be an elementary abelian group isomorphic to $(\Z / 2 \Z)^k$, with $k\leq d$. This is shown in Proposition~\ref{prop:iso-structure}, which was proved by Ikeda for compact manifolds in \cite{Ikeda1985}, but the bound $k\leq d$ is due to the affine structure of the isometry group, hence it is new to our knowledge and sharp, see Example \ref{ex:rectangle}.

What genuinely obstructs the
argument is the possible presence of eigenfunctions of vanishing mean, and this
is exactly what Theorem~\ref{thm:main} removes for a generic domain.

\begin{theorem}\label{thm:isometry-generic}

Let $D \subset \R^d$ be a bounded connected open set of class $C^m$
($m \geq 3$, $d \geq 2$). Then the set of domains $\Omega\in\OO^m(D)$ such that
\[
\mathrm{Iso}(\Omega)=\{\mathrm{Id}\}
\]
is residual in $\OO^m(D)$.
\end{theorem}

Although Theorem~\ref{thm:isometry-generic} follows from the spectral result by
a short argument, we state it as a main result because it gives an intrinsic
spectral-geometric payoff: the same condition that makes the constant function
visible to every eigenmode also forbids hidden Euclidean symmetries of the
domain.

\subsection{Meagerness of analytic boundaries}\label{subsec:meager-intro}

{The proof of Theorem~\ref{thm:main} leans on a third result, of a different nature, which is the pivot of our approach.}

\begin{theorem}\label{thm:meager}
{Let $D \subset \R^d$ be a bounded connected open set of class $C^m$
($m \geq 3$, $d \geq 2$). The set of domains $\Omega \in \OO^m(D)$ whose
boundary is a real-analytic hypersurface is meager in $\OO^m(D)$.}
\end{theorem}

{The role of Theorem~\ref{thm:meager} in this paper is to dispose of the degenerate case of the shape derivative, described in the next subsection. Its proof, given in Section~\ref{sec:meager}, is elementary: a normal-graph chart reduces the analyticity of a perturbed boundary to that of a scalar function on a fixed compact manifold, and a Baire-category argument shows that analytic functions, and in fact $C^{m+1}$ functions already, form a meager subset of $C^m$. The statement itself, however, does not seem to have been recorded in the literature. Since it carries the proof of Theorem~\ref{thm:main} and is independent of the particular shape functional under study, hence reusable for other generic properties of spectral shape functionals, we single it out as a standalone result.}

{Theorem~\ref{thm:meager} also quantifies a natural intuition. Domains with analytic boundary are dense in $\OO^m(D)$, by Whitney's theorem on analytic approximation of submanifolds \cite{Whitney1936}, yet a generic $C^m$ domain is not analytic, in the same way as a generic continuous function is nowhere differentiable.}

\subsection{Relation with previous work and shape-derivative degeneracy}\label{subsec:previous}

The proof of Theorem~\ref{thm:main} is based on a shape-derivative argument. If
$\psi$ solves
\[
-\Delta\psi-\lambda_n\psi=1,\qquad \psi|_{\partial\Omega}=0,
\]
then differentiating the average $\Omega\mapsto\int_\Omega\varphi_n$ leads to
the boundary kernel
\[
K_\Omega = \partial_\nu\varphi_n\,\partial_\nu\psi.
\]
When $K_\Omega$ does not vanish identically, one can perturb the domain so as to
make the mean nonzero. The degenerate case $K_\Omega\equiv0$ can be interpreted
as an overdetermined elliptic configuration: the auxiliary function $\psi$ has
both zero Dirichlet and zero Neumann data on the boundary. This is reminiscent of
Schiffer-type phenomena related to Serrin's theorem and to the Schiffer
conjecture \cite{Serrin1971,Berenstein1980,Dalmasso1999,Yau1982,FallMinlendWeth2024}. 

In our proof, we use the structural consequence that such
a degeneracy forces the boundary to be real-analytic, by the regularity theorem
of Kinderlehrer and Nirenberg \cite{KinderlehrerNirenberg1977}. {Theorem~\ref{thm:meager} then disposes of this exceptional case in the Baire-category argument.}

{Let us now describe precisely the relation with the existing literature. The question was recently raised explicitly by Steinerberger and Venkatraman \cite{SteinerbergerVenkatraman2025}. An affirmative generic answer is in fact contained, implicitly, in the proof of a broader theorem of Chitour, Coron and Garavello on controllability by spatially homogeneous boundary controls \cite{ChitourCoronGaravello2006}: the content of Theorem~\ref{thm:main} can be extracted from the proof of \cite[Theorem~2]{ChitourCoronGaravello2006}, in the regularity and topological framework adapted to their control problem, but it is not isolated there as a standalone spectral-geometric statement. One aim of the present paper is to extract this phenomenon, make it directly citable in the Courant--Micheletti space $\OO^m(D)$, and record its geometric and control-theoretic consequences. Neither the analysis of the sign characters of Euclidean symmetries and the generic triviality of $\mathrm{Iso}(\Omega)$ (Theorem~\ref{thm:isometry-generic}), nor the meagerness of analytic boundaries (Theorem~\ref{thm:meager}), nor the positive controllability and stabilization results of Section~\ref{sec:control}, appear in \cite{ChitourCoronGaravello2006}.}

For Theorem~\ref{thm:main} itself, the genuine novelty lies in the proof. Both the
present argument and that of \cite{ChitourCoronGaravello2006} reduce the problem
to the possible complete degeneracy of the first shape derivative, namely
$K_\Omega \equiv 0$. The approaches then diverge. In
\cite{ChitourCoronGaravello2006}, this degeneracy is excluded on each individual
domain by a slicing argument exploiting the simplicity of the eigenvalue. Here,
instead, we interpret $K_\Omega \equiv 0$ as an overdetermined elliptic problem
of Schiffer type, which implies analyticity of the boundary by the theorem of
Kinderlehrer and Nirenberg. Theorem~\ref{thm:meager} then removes these
exceptional domains by Baire category. Although the pointwise exclusion obtained
in \cite{ChitourCoronGaravello2006} is a stronger intermediate statement, the
present approach reaches the same generic conclusion through a mechanism that is
independent of the particular functional
$\Omega \mapsto \int_\Omega \varphi_n$ and therefore potentially applicable to
other generic spectral shape problems. We return to this comparison in
Remark~\ref{rem:fork}, at the precise point where the two proofs diverge.

A statement equivalent to
Theorem~\ref{thm:main} is asserted by McDonald and Meyers
\cite[Theorem~4.1]{McDonaldMeyers2003}, in the context of the moment spectrum,
the family of pairings $\int_\Omega \varphi_n$ that govern the coupling of each
mode to a spatially uniform source. Their proof proceeds by shape differentiation
and reaches an integral identity, but the final transversality step, in the
paragraph following their~(4.7), infers the non-vanishing of that integral from
the non-vanishing of a two-variable kernel, and thereby overlooks a possible
cancellation occurring after integration over the boundary. We
provide a complete and self-contained proof in which this degeneracy is
identified and then bypassed through Theorem~\ref{thm:meager}.

\subsection{Applications}

Beyond its spectral-geometric interest, Theorem~\ref{thm:main} settles a natural
question in the control of parabolic equations, developed in
Section~\ref{sec:control}. Consider the heat equation on $\Omega$ driven by a
single scalar control acting through the fixed, spatially homogeneous profile
$\mathbf{1}_\Omega$,
\[
\partial_t z = \Delta z + f(t)\,\mathbf{1}_\Omega, \qquad
z|_{\partial\Omega} = 0, \qquad z(0) = z_0.
\]
{For such a degenerate rank-one input, the spectral coefficients of the control profile determine which modes can be reached. In particular, the natural notion considered here is approximate controllability rather than the stronger null-controllability properties available for genuinely distributed controls.} By the
Fattorini--Hautus criterion, in the form established by Badra and Takahashi
\cite{Fattorini1966, BadraTakahashi2014}, the system
is approximately controllable if and only if
$\int_\Omega \varphi_n \neq 0$ for every $n$, which is exactly the condition
supplied by Theorem~\ref{thm:main}. We thus obtain that, for a generic
domain, the Dirichlet heat equation is approximately controllable, and
{rapidly stabilizable (at any prescribed exponential decay rate)}, by a single spatially homogeneous scalar control
(Corollary~\ref{cor:control} and Corollary~\ref{cor:rapid-stabilization}). The same condition appears for a spatially homogeneous scalar boundary control,
through Green's formula. Consequently, this stands in contrast with a sharp negative
result of Chitour, Coron and Garavello \cite{ChitourCoronGaravello2006} for the
stronger notion of steady-state controllability, and places this theorem within a
sustained program on domain-dependent controllability, to which
Section~\ref{sec:control} returns. 

The same spectral coefficients also arise in photonics, which was in fact the
original motivation of Steinerberger and Venkatraman
\cite{SteinerbergerVenkatraman2025}. In the analysis of the epsilon-near-zero
resonant cavities introduced by Liberal, Mahmoud and Engheta \cite{LiberalMahmoudEngheta2016}, Kohn and Venkatraman \cite{KohnVenkatraman2024} showed that
the limiting resonances are governed by the Dirichlet eigenfunctions with
nonzero average. Consequently, while it is known that infinitely many such
modes always exist, Theorem~\ref{thm:main} implies that, for a generic
dielectric inclusion, every Dirichlet mode contributes, so no admissible
resonance is lost.

\subsection{Organization of the paper}

{Section~\ref{sec:setting} introduces the Micheletti topology, recalls the needed
analytic perturbation facts, performs the shape-derivative computation, and
proves Theorem~\ref{thm:main}, taking Theorem~\ref{thm:meager} for granted.
Section~\ref{sec:meager} proves Theorem~\ref{thm:meager}.
Section~\ref{sec:isometries} proves the spectral criterion for absence
of Euclidean symmetries and deduces Theorem~\ref{thm:isometry-generic}. Finally, Section~\ref{sec:control}
collects the applications to scalar parabolic control.}
}

\section{{Setting and proof of the main theorem}}\label{sec:setting}
\subsection{Setting}
Let $d\geq 2$ and $D \subset \R^d$ be a bounded connected open set of class $C^m$ with $m \geq 3$. We denote by $\text{Diff}^m(\R^d)$ the group of $C^m$-diffeomorphisms of $\R^d$ and by $C_0^m(\R^d,\R^d)$ the space of $C^m$-vector fields on $\R^d$ that vanish at infinity for the $C^m$ norm, defined by
\begin{equation*}
    \norm{f}_{C^m} := \sum_{|\alpha| \leq m} \sup_{x \in \R^d} |\partial^\alpha f(x)|.
\end{equation*}
We will also define 
\begin{equation*}
    \text{Diff}^m_0(\R^d) := \{ f \in \text{Diff}^m(\R^d) : f - \text{id} \in C_0^m(\R^d,\R^d) \},
\end{equation*}
which is a subgroup of $\text{Diff}^m(\R^d)$. We now define $\OO^m(D)$ as the set of all domains $\Omega \subset \R^d$ such that $\Omega=f(D)$ for some $f \in \text{Diff}^m_0(\R^d)$. We endow it with the Courant--Micheletti distance, defined by
{\begin{equation*}
    d(\Omega_1,\Omega_2) := \inf_{f_i \in \text{Diff}^m_0(\R^d),\, N\geq 1} \{ \sum_{i=1}^N \norm{f_i - \text{id}}_{C^m}+\norm{f_i^{-1} - \text{id}}_{C^m} : \, f_N \circ \cdots \circ f_1(\Omega_1) = \Omega_2 \}.
\end{equation*}}
This distance was first introduced by Micheletti \cite{Micheletti1972}, where it is also proved that $\OO^m(D)$ is complete and separable. We set $\Theta^m = C^m(\overline{D};\R^d)$. The map 
\begin{equation*}
\theta \in \Theta^m \mapsto (\mathrm{id} + \theta)(D) \in \OO^m(D),
\end{equation*}
is well defined for $\theta$ small enough, and it is continuous with a local right inverse near $D$, see Henry \cite{Henry2005}. Since $\OO^m(\Omega)=\OO^m(D)$ for every $\Omega \in \OO^m(D)$, we may speak of smooth or analytic maps locally defined on $\OO^m(D)$ by considering their pullback under the above map.

\medskip

For $\Omega \in \OO^m(D)$, we denote by $0<\lambda_1(\Omega) \leq \lambda_2(\Omega) \leq \cdots$ the eigenvalues of the Dirichlet Laplacian on $\Omega$, and we define the set of domains with simple $n$-th eigenvalue as
\[
\mathcal{S}_n := \{ \Omega \in \OO^m(D) : \lambda_n(\Omega) \text{ is simple} \}, \qquad \mathcal{S} := \bigcap_{n \geq 1} \mathcal{S}_n.
\]
It is well known that $\mathcal{S}_n$ is open and dense in $\OO^m(D)$, see Uhlenbeck \cite{Uhlenbeck1976} and Micheletti \cite{Micheletti1972}, hence $\mathcal{S}$ is residual. Moreover, for every $n \geq 1$ and every $\Omega \in \mathcal{S}_n$, setting 
$\Omega_\theta :=(\mathrm{id} + \theta)(\Omega)$ for small $\theta \in C^m(\overline{\Omega};\R^d)$ (so that $\Omega_\theta \in \mathcal{S}_n$, as $\mathcal{S}_n$ is open), 
the maps $\theta \mapsto \lambda_n(\Omega_\theta)$ and $\theta \mapsto \varphi_n(\Omega_\theta)\circ (\mathrm{id}+\theta) \in H^2(\Omega)\cap H^1_0(\Omega)$ are analytic near $\theta = 0$, where $\varphi_n(\Omega_\theta)$ is an $L^2$-normalized eigenfunction associated with $\lambda_n(\Omega_\theta)$, again by \cite{Henry2005}.
\begin{remark}\label{rem:continuity}
    On $\mathcal{S}_n$, the map $\Omega \mapsto \int_\Omega \varphi_n(\Omega)$ is not well defined, since the eigenfunction $\varphi_n(\Omega)$ is only determined up to sign. The map $M_n : \Omega \mapsto \left|\int_\Omega \varphi_n(\Omega)\right|$, however, is well defined on $\mathcal{S}_n$ and continuous, by the analytic dependence of the eigenpair recalled above.
\end{remark}
\paragraph{Proof strategy} For each $n\geq 1$, we define 
\begin{equation*}
    Z_n := \{ \Omega \in \mathcal{S}_n : \int_\Omega \varphi_n(\Omega) = 0 \}.
\end{equation*}
We will show that $Z_n$ is closed in $\mathcal{S}_n$ and has empty interior. Since $\mathcal{S}_n$ is open and dense in $\OO^m(D)$, the set $\mathcal{S}_n \setminus Z_n$ is then open and dense, and Baire's theorem implies that $\mathcal{R} = \bigcap_{n \geq 1} (\mathcal{S}_n \setminus Z_n)$ is residual in $\OO^m(D)$.
To prove that $Z_n$ has empty interior, we use a shape-derivative argument: we show that when the derivative vanishes in every direction, {the dual function} $\psi$ {(introduced in Lemma~\ref{lem:psi} below)} solves an overdetermined elliptic problem, and a theorem of Kinderlehrer and Nirenberg \cite{KinderlehrerNirenberg1977} then forces the boundary to be analytic. {Theorem~\ref{thm:meager}, proved in Section~\ref{sec:meager}, then ensures that these domains form a meager subset of $\OO^m(D)$, which completes the argument.}

\subsection{Shape derivative computation}\label{subsec:shape}
The aim of this subsection is to compute the shape derivative of the map $\Omega \mapsto \int_\Omega \varphi_n(\Omega)$ and to express it as a boundary integral. This yields a simple geometric criterion for the degeneracy of the derivative, namely the vanishing of a kernel $K_\Omega$ on $\partial\Omega$, on which the whole argument rests.

In this subsection we fix $\Omega \in Z_n$ and $V\in C^m(\overline{\Omega};\R^d)$. Then, we set 
\[
\Omega_t := T_t(\Omega), \quad T_t := \mathrm{id} + tV, \quad t \in (-\varepsilon, \varepsilon),
\]
with $\varepsilon >0$ small enough such that $T_t$ can be extended to $\R^d$ and $T_t \in \text{Diff}^m_0(\R^d)$, for every $t \in (-\varepsilon, \varepsilon)$. 
Hence on this interval, we have $\Omega_t \in \mathcal{S}_n$ and we can define $\lambda_n(t) := \lambda_n(\Omega_t)$ and $\varphi_n(t) := \varphi_n(\Omega_t)\circ T_t \in H^2(\Omega)\cap H^1_0(\Omega)$, which are analytic functions of $t$ (for $\varphi_n(t)$, as it is only determined up to sign, this means that we fix locally an analytic branch).

Denote by $E : H^1_0(\Omega) \to H^1(\R^d)$ the extension-by-zero operator, set $\tilde{\varphi}_n(t) := E(\varphi_n(t)) \in H^1(\R^d)$, and let $u_n(t) \in L^2(\R^d)$ be the extension by zero of $\varphi_n(\Omega_t)$. 
Since the map
\begin{equation*}
    (t,w) \in (-\varepsilon, \varepsilon) \times H^1(\R^d) \mapsto w \circ T_t^{-1} \in L^2(\R^d),
\end{equation*}
is $C^1$, see \cite[Chapter 3]{Henry2005}, and since $u_n(t)= \tilde{\varphi}_n(t) \circ T_t^{-1}$, the map $t \mapsto u_n(t) \in L^2(\R^d)$ is $C^1$, and we may define 
\begin{align*}
    \dot \varphi_n &= \frac{\df}{\df t}\bigr|_{t=0} \tilde{\varphi}_n(t) \in H^1(\R^d), \\
    \varphi_n' &= \frac{\df}{\df t}\bigr|_{t=0} u_n(t) \in L^2(\R^d).
\end{align*}
Using $\frac{\df}{\df t}\bigr|_{t=0} T_t^{-1}=-V$ and the chain rule, a standard computation gives
\begin{equation}\label{eq:form_d}
    \varphi'_n = \dot \varphi_n - \nabla \tilde{\varphi}_n \cdot V.
\end{equation}

We may now compute the shape derivative of the mean. Since $\varphi_n$ vanishes on $\partial\Omega$, the boundary term in Hadamard's formula for the derivative of a domain integral drops out, and the latter reduces to a pure interior contribution, see Henrot and Pierre \cite[Chapter 5]{HenrotPierre2018}:
\begin{equation}\label{eq:derivative_mean}
\frac{\df}{\df t}\bigr|_{t=0} \int_{\Omega_t} \varphi_n(\Omega_t)= \int_\Omega \varphi_n' .
\end{equation}

\begin{proposition}\label{prop:phi_prime}
Under the above notations, we have $\varphi_n' \in H^1(\Omega)$ and 
\begin{equation}\label{eq:phi_prime}
\begin{cases}
-\Delta \varphi_n' - \lambda_n \varphi_n' = \lambda_n' \, \varphi_n & \text{in } \Omega, \\
\varphi_n'\bigr|_{\partial\Omega} = -(V \cdot \nu)\, \partial_\nu \varphi_n & \text{on } \partial\Omega, \\
\int_\Omega \varphi_n \varphi_n' =0,
\end{cases}
\end{equation}
where $\nu$ is the outward unit normal to $\partial\Omega$.
\end{proposition}
\begin{proof}
    From \eqref{eq:form_d}, since $\varphi_n \in H^2(\Omega)$ and ${\nabla \tilde{\varphi}_n}_{|\Omega} = \nabla \varphi_n$, we have $\varphi_n' \in H^1(\Omega)$. Taking the boundary trace of \eqref{eq:form_d} and using that $\nabla \varphi_n =\partial_\nu \varphi_n \, \nu$ on $\partial\Omega$ (as $\varphi_n$ vanishes there), we obtain
    \begin{equation*}
        {\varphi_n'}_{|\partial\Omega} = -(V \cdot \nu)\, \partial_\nu \varphi_n.
    \end{equation*}
    Finally, since $u_n(t)$ is an eigenfunction of the Dirichlet Laplacian on $\Omega_t$ with eigenvalue $\lambda_n(t)$, for all $\phi \in C^\infty_c(\Omega)$ and $|t|$ small enough that $\mathrm{supp} \, \phi \subset \Omega_t$, we have
    \begin{equation*}
        \scalp{u_n(t)}{\Delta \phi}_{L^2(\R^d)} +\lambda_n(t) \scalp{u_n(t)}{\phi}_{L^2(\R^d)} = 0.
    \end{equation*}
    Differentiating at $t=0$ gives the first line of \eqref{eq:phi_prime}. Since $u_n(t)$ is $L^2$-normalized for all $t$, differentiating $\norm{u_n(t)}_{L^2(\R^d)}^2=1$ at $t=0$ gives the last line.
\end{proof}
In order to turn the interior integral \eqref{eq:derivative_mean} into a boundary integral in $V \cdot \nu$, we introduce an adjoint (dual) function $\psi$. Its existence is not automatic: by the Fredholm alternative, the forcing term $1$ must be orthogonal to $\ker(-\Delta - \lambda_n) = \mathrm{span}\,\varphi_n$, that is, $\int_\Omega \varphi_n = 0$. This is precisely the assumption $\Omega \in Z_n$, and it is the only place in the argument where it is used.
\begin{lemma}[Dual function]\label{lem:psi}
The problem
\begin{equation}\label{eq:psi}
\begin{cases}
-\Delta \psi - \lambda_n \psi = 1 & \text{in } \Omega, \\
\psi = 0 & \text{on } \partial\Omega, \\
\int_\Omega \varphi_n \psi \df x = 0
\end{cases}
\end{equation}
admits a unique solution $\psi \in H^2(\Omega) \cap H^1_0(\Omega)$. Furthermore, $\psi \in C^{2,\alpha}(\overline\Omega)$ for all $\alpha \in (0,1)$.
\end{lemma}
\begin{proof}
    By the Fredholm alternative, the first two lines of \eqref{eq:psi} admit a solution if and only if {$1 \perp \ker(-\Delta - \lambda_n) = \mathrm{span}\, \varphi_n$}, which is equivalent to $\int_\Omega \varphi_n = 0$. 
    Hence, as $\Omega \in Z_n$, there exists a solution $\psi \in H^2(\Omega) \cap H^1_0(\Omega)$ to the first two lines of \eqref{eq:psi}. Now, as $\lambda_n$ is simple, the kernel of $-\Delta - \lambda_n$ is one-dimensional and generated by $\varphi_n$, hence the condition $\int_\Omega \varphi_n \psi \df x = 0$ ensures the uniqueness of the solution. 
    Finally, as $\Omega$ is of class $C^m$ with $m\geq 3$, we have $\psi \in C^{2,\alpha}(\overline\Omega)$ for all $\alpha \in (0,1)$ by the classical Schauder estimates of Gilbarg and Trudinger \cite[Theorem 6.15]{gilbarg2001elliptic}.
\end{proof}
\begin{proposition}\label{prop:key_formula}
Let $\psi$ be the dual function given by Lemma \ref{lem:psi}. Then we have
\begin{equation}\label{eq:key_formula}
\frac{\df}{\df t}\bigr|_{t=0} \int_{\Omega_t} \varphi_n(\Omega_t) = \int_{\partial\Omega} (V \cdot \nu)\, \partial_\nu \varphi_n \, \partial_\nu \psi \df \sigma.
\end{equation}
\end{proposition}
\begin{proof}
    Using \eqref{eq:derivative_mean} and the first line of \eqref{eq:psi}, we have
    \begin{equation*}
        \frac{\df}{\df t}\bigr|_{t=0} \int_{\Omega_t} \varphi_n(\Omega_t) = \int_\Omega (-\Delta \psi - \lambda_n \psi) \varphi_n' .
    \end{equation*}
    Then, applying Green's formula, using the fact that $\psi$ satisfies the first two lines of \eqref{eq:psi}, and using \eqref{eq:phi_prime}, we get
    \begin{align*}
       \frac{\df}{\df t}\bigr|_{t=0} \int_{\Omega_t} \varphi_n(\Omega_t) &=\int_{\Omega} \psi (-\Delta \varphi_n' - \lambda_n \varphi_n') + \int_{\partial\Omega} \psi \partial_\nu \varphi_n' - \varphi_n' \partial_\nu \psi \df \sigma,\\
       &= - \int_{\partial\Omega} \varphi_n' \partial_\nu \psi \df \sigma,\\
        &= \int_{\partial\Omega} (V \cdot \nu)\, \partial_\nu \varphi_n \, \partial_\nu \psi \df \sigma.
    \end{align*}
\end{proof}

Now that the shape derivative is given by the explicit boundary integral \eqref{eq:key_formula}, we look for a simple criterion for its vanishing across all admissible perturbations. Since the perturbation enters only through the scalar $V \cdot \nu$ on $\partial\Omega$, this criterion reduces to a purely geometric quantity on $\partial\Omega$, namely the kernel
\begin{equation}\label{eq:K_Omega}
    K_\Omega := \partial_\nu \varphi_n \, \partial_\nu \psi \in C^1(\partial\Omega).
\end{equation}
The regularity of $K_\Omega$ is a direct consequence of Schauder estimates, as $\varphi_n$ and $\psi$ both lie in $C^{2,\alpha}(\overline\Omega)$ for all $\alpha \in (0,1)$.

The following corollary makes precise the sense in which $K_\Omega$ captures the degeneracy of the shape derivative.
\begin{corollary}\label{cor:K-zero}
We have $K_\Omega \equiv 0$ if and only if $\frac{\df}{\df t}\bigr|_{t=0} \int_{\Omega_t} \varphi_n(\Omega_t) = 0$ for all $V \in C^m(\overline{\Omega};\R^d)$.
\end{corollary}
\begin{proof}
Thanks to \eqref{eq:key_formula}, we have 
\begin{equation*}
\frac{\df}{\df t}\bigr|_{t=0} \int_{\Omega_t} \varphi_n(\Omega_t) = \int_{\partial\Omega} (V \cdot \nu)\, K_\Omega \df \sigma.
\end{equation*}
Hence if $K_\Omega \equiv 0$, then $\frac{\df}{\df t}\bigr|_{t=0} \int_{\Omega_t} \varphi_n(\Omega_t) = 0$ for all $V \in C^m(\overline{\Omega};\R^d)$. Conversely, assume this derivative vanishes for all $V \in C^m(\overline{\Omega};\R^d)$, and let ${\chi} \in C^\infty_c(\R^d)$ be a bump function whose support meets $\partial\Omega$. Since $\partial\Omega$ is of class $C^m$ with $m \geq 3$, we have $\nu \in C^{m-1}(\partial\Omega;\R^d)$, and we extend it to a $C^{m-1}$ vector field $\tilde{\nu}$ on $\R^d$, supported in a neighborhood of $\overline{\Omega}$ and equal to $\nu$ on $\partial\Omega$.

Let $\rho_\varepsilon$ be a standard mollifier and set $V_\varepsilon := {\chi} \left(\rho_\varepsilon * \tilde{\nu}\right)$. Then ${V_\varepsilon}_{|\overline{\Omega}} \in C^m(\overline{\Omega};\R^d)$ and $V_\varepsilon \cdot \nu \to {\chi}$ uniformly on $\partial\Omega$ as $\varepsilon \to 0$, so that
\begin{equation*}
    0 = \int_{\partial\Omega} (V_\varepsilon \cdot \nu)\, K_\Omega \df \sigma \underset{\varepsilon \rightarrow 0}{\longrightarrow} \int_{\partial\Omega} {\chi}\, K_\Omega \df \sigma.
\end{equation*} 
Thus $\int_{\partial\Omega} {\chi}\, K_\Omega \df \sigma = 0$ for every such ${\chi}$, hence $K_\Omega \equiv 0$.
\end{proof}

\subsection{Overdetermined elliptic problem}

At this point one might hope for a direct argument that $Z_n$ has empty interior: it would suffice to show that, for every $\Omega \in Z_n$, the shape derivative is nonzero in some direction, that is, $K_\Omega \not\equiv 0$. 
This, however, is not the route we follow. Indeed, we show in Lemma~\ref{lem:psi-overdetermined} that $K_\Omega \equiv 0$ forces the dual function $\psi$ to solve an \emph{overdetermined} elliptic problem with constant forcing, with both Dirichlet and Neumann data vanishing on $\partial\Omega$. This is reminiscent of Schiffer-type problems, see \cite{Yau1982,Dalmasso1999,Williams1981}, but it is used here only as a way to read the degenerate case of the shape derivative. The structural information we need is simply that $K_\Omega \equiv 0$ forces $\partial\Omega$ to be real-analytic (Proposition~\ref{prop:analytic_boundary}). {Theorem~\ref{thm:meager} shows that domains with analytic boundary form a meager subset of $\OO^m(D)$: they cannot fill any open set, and the density argument goes through all the same.}

\begin{lemma}\label{lem:psi-overdetermined}
    Let $\Omega \in Z_n$ such that $K_\Omega$ defined by \eqref{eq:K_Omega} is identically zero. Then $\psi$ is a solution to the overdetermined elliptic problem
    \begin{equation}\label{eq:psi-overdetermined}
\begin{cases}
-\Delta \psi - \lambda_n \psi = 1 & \text{in } \Omega, \\
\psi = 0 & \text{on } \partial\Omega, \\
\partial_\nu \psi = 0 & \text{on } \partial\Omega.
\end{cases}
\end{equation}
\end{lemma}
\begin{proof}
First, we show that $J\subset \partial\Omega$, the set where $\partial_\nu \varphi_n$ vanishes, has empty interior.
Assume by contradiction that $J$ has non-empty interior. 
{On a relatively open boundary patch contained in $J$, the eigenfunction has both Cauchy data $\varphi_n=0$ and $\partial_\nu\varphi_n=0$. 
By boundary unique continuation for elliptic equations, see Isakov \cite[Theorem 3.3.1]{Isakov2017}, this forces $\varphi_n$ to vanish in all $\Omega$, which is a contradiction.} Hence, as $K_\Omega$ vanishes identically, $\partial_\nu \psi = 0$ on a dense subset of $\partial\Omega$, and by continuity of $\partial_\nu \psi$ we conclude $\partial_\nu \psi = 0$ on all of $\partial\Omega$.
\end{proof}

\begin{proposition}\label{prop:analytic_boundary}
    Let $\Omega \in Z_n$ such that $K_\Omega$ defined by \eqref{eq:K_Omega} is identically zero. Then $\Omega$ has an analytic boundary.
\end{proposition}
\begin{proof}
    By Lemma \ref{lem:psi-overdetermined}, $\psi$ solves the overdetermined elliptic problem \eqref{eq:psi-overdetermined}, and it is of class $C^2(\overline\Omega)$. 
    {The statement of \cite[Theorem~1$'$]{KinderlehrerNirenberg1977} is local, so fix $x_0 \in \partial\Omega$. The function $\psi$ satisfies $F(x, \psi, \nabla \psi, \nabla^2 \psi)=0$ in $\Omega$, with
    \[
    F(x,z,p,M):=\operatorname{tr}M+\lambda_n z+1.
    \]
    This operator is uniformly elliptic and analytic in all its arguments. Near $x_0$, the hypotheses of \cite[Theorem~1$'$]{KinderlehrerNirenberg1977} are satisfied: the boundary $\partial\Omega$ is a $C^1$ hypersurface, $\psi\in C^2(\overline\Omega)$ locally near $x_0$, and $\psi$ has zero Cauchy data on $\partial\Omega$ by \eqref{eq:psi-overdetermined}. Finally, the nondegeneracy condition (1.11) of \cite{KinderlehrerNirenberg1977} holds because
    \[
    F(x_0,0,0,0)=1\neq0.
    \]
    Assertion (iii) of \cite[Theorem~1$'$]{KinderlehrerNirenberg1977} therefore yields that $\partial\Omega$ is real-analytic in a neighborhood of $x_0$. Since $x_0\in\partial\Omega$ was arbitrary, $\partial\Omega$ is real-analytic.}
\end{proof}
\begin{remark}
    This proposition is where the hypothesis $m \geq 3$ is used. 
    Theorem 1' of \cite{KinderlehrerNirenberg1977} requires the solution to be of class $C^2$ up to the boundary. Since $\partial\Omega \in C^m$ with $m \geq 3$, 
    in particular $\partial\Omega \in C^{2,\alpha}$, the {elliptic regularity} of Lemma~\ref{lem:psi} {yields} $\psi \in C^{2,\alpha}(\overline\Omega) \subset C^2(\overline\Omega)$, exactly the regularity needed. A merely $C^2$ boundary (i.e. $m = 2$) would in general only provide $\psi \in C^{1,\alpha}(\overline\Omega)$, which falls short of the hypotheses of \cite{KinderlehrerNirenberg1977}.
\end{remark}

\begin{remark}[Comparison with \cite{ChitourCoronGaravello2006}]\label{rem:fork}
{It is precisely at this stage that the present proof departs from that of \cite{ChitourCoronGaravello2006}. There, the analogue of the configuration $K_\Omega \equiv 0$ is ruled out on every domain with simple $n$-th eigenvalue, through a slicing argument. Here it is not ruled out on any fixed domain: Proposition~\ref{prop:analytic_boundary} converts it into boundary analyticity, and Theorem~\ref{thm:meager} shows that this can only occur on a meager set of domains.}
\end{remark}

\subsection{Conclusion}
We now have all the ingredients to prove Theorem~\ref{thm:main}. The shape-derivative computation of Section~\ref{subsec:shape} reduces the empty-interior property of $Z_n$ to a dichotomy governed by the kernel $K_\Omega$. {The only nontrivial input left is the meagerness of domains with analytic boundary, Theorem~\ref{thm:meager}, which we take for granted here and prove in Section~\ref{sec:meager}.}

\begin{proof}[Proof of Theorem \ref{thm:main}]
Fix $n \geq 1$. As explained above, it suffices to show that $Z_n$ is closed in $\mathcal{S}_n$ and has empty interior. 
Closedness is immediate, since $Z_n = M_n^{-1}(\{0\})$ with $M_n$ the continuous map of Remark \ref{rem:continuity}.

We now show that $Z_n$ has empty interior. Let $\Omega \in Z_n$ and denote by $\mathcal{A} \subset \mathcal{O}^m(D)$ the set of domains with analytic boundary. We distinguish two cases.
\begin{itemize}
 \item If $\Omega \notin \mathcal{A}$, then by the contrapositive of Proposition \ref{prop:analytic_boundary} the kernel $K_\Omega$ of \eqref{eq:K_Omega} is not identically zero. By Corollary \ref{cor:K-zero}, there exists $V \in C^m(\overline{\Omega};\R^d)$ such that $\frac{\df}{\df t}\bigr|_{t=0} \int_{\Omega_t} \varphi_n(\Omega_t) \neq 0$. Since $t \mapsto \int_{\Omega_t} \varphi_n(\Omega_t)$ is analytic, for $t \neq 0$ small enough we have
 \begin{equation*}
     \int_{\Omega_t} \varphi_n(\Omega_t) = t \frac{\df}{\df s}\bigr|_{s=0} \int_{\Omega_s} \varphi_n(\Omega_s) + O(t^2) \neq 0.
 \end{equation*}
 As $t\mapsto \Omega_t$ is continuous for the topology of $\OO^m(D)$, every open neighborhood $U$ of $\Omega$ in $\mathcal{S}_n$ contains some $\Omega_t$ with $\int_{\Omega_t} \varphi_n(\Omega_t) \neq 0$, hence $\Omega_t \notin Z_n$.
 \item If $\Omega \in \mathcal{A}$, let $U$ be an open neighborhood of $\Omega$ in $\mathcal{S}_n$. Since $\mathcal{A}$ is meager in $\OO^m(D)$ by {Theorem~\ref{thm:meager}}, there exists $\Omega' \in U$ with $\Omega' \notin \mathcal{A}$. If $\Omega' \notin Z_n$, we are done. Otherwise $\Omega' \in Z_n \setminus \mathcal{A}$ falls into the first case{: applied to $\Omega'$ and to the open set $U$, which is also a neighborhood of $\Omega'$, it provides $\Omega'' \in U$ with $\Omega'' \notin Z_n$}.
\end{itemize}
Thus $Z_n$ has empty interior in $\mathcal{S}_n$. Consequently $\mathcal{S}_n \setminus Z_n$ is open and dense in $\OO^m(D)$, and Baire's theorem implies that $\mathcal{R} = \bigcap_{n \geq 1} (\mathcal{S}_n \setminus Z_n)$ is residual in $\OO^m(D)$.
\end{proof}

\section{{Meagerness of analytic boundaries}}\label{sec:meager}
{This section is devoted to the proof of Theorem~\ref{thm:meager}: the set $\mathcal{A} \subset \OO^m(D)$ of domains with real-analytic boundary is meager.} The argument is local. It proceeds in two steps. First, in a normal-graph chart around a domain with analytic boundary, the analyticity of the perturbed boundary $\partial\Omega(\sigma)$ is equivalent to the analyticity of the scalar graph function $\sigma$ (Lemma~\ref{lem:param-ana}). 
This transfers the problem to a statement about scalar functions on a fixed compact manifold. Second, on any smooth compact manifold $M$, the space $C^{m+1}(M)$---and \emph{a fortiori} $C^\omega(M)$---is meager in $C^m(M)$ (Lemma~\ref{lem:Cm-meager}).

\medskip

\subsection{Reduction to a scalar graph function}

Let $\Omega \in \OO^m(D)$ have analytic boundary. Then the exterior normal satisfies $\nu \in C^\omega(\partial\Omega;\R^d)$ and admits a smooth extension to $\R^d$, and by \cite[Appendix 2]{Henry2005}, for $\sigma \in C^m(\partial\Omega)$ small enough in $C^m$ norm we may define a domain $\Omega(\sigma)$ with boundary
\begin{equation*}
    \partial \Omega(\sigma) := \{ x + \sigma(x)\, \nu(x) : x \in \partial\Omega \}.
\end{equation*}
Moreover, the map $\sigma \mapsto \Omega(\sigma)$, defined in a neighborhood of $0$ in $C^m(\partial\Omega)$, is a homeomorphism onto a neighborhood of $\Omega$ in $\OO^m(D)$.

\begin{lemma}\label{lem:param-ana}
There exists $\varepsilon >0$ small enough such that for every $\sigma \in C^m(\partial\Omega)$ with $\norm{\sigma}_{C^m} < \varepsilon$, we have
\begin{equation*}
    {\partial \Omega(\sigma) \text{ is real-analytic}} \iff \sigma \in C^\omega(\partial\Omega).
\end{equation*}
\end{lemma}
\begin{proof}
    Let $\varepsilon >0$ be small enough such that the map $\sigma \in C^m(\partial\Omega) \mapsto \Omega(\sigma) \in \OO^m(D)$ is well defined on a $C^m$ ball around $0$ of radius $\varepsilon$, and is a homeomorphism onto a neighborhood of $\Omega$ in $\OO^m(D)$. 
    \medskip

    We first show that if $\sigma \in C^\omega(\partial\Omega)$, then $\partial \Omega(\sigma)$ is analytic. Given $\sigma \in C^\omega(\partial \Omega)$ with {$\norm{\sigma}_{C^m} < \varepsilon$}, define $F : \partial\Omega \to \R^d$ by 
    \begin{equation*}
       \forall x \in \partial\Omega, \quad F(x) := x + \sigma(x)\, \nu(x).
    \end{equation*}
    It suffices to show that $F$ is an injective analytic immersion: since $\partial \Omega$ is compact, this implies that $F$ is an analytic embedding, so that $\partial \Omega(\sigma) = F(\partial\Omega)$ is an analytic submanifold of $\R^d$. 
    For $v\in T_x \partial \Omega$, we have
    \begin{equation*}
        \mathrm{d}F_x (v) = v + \sigma(x)\, \mathrm{d}\nu_x (v) + \mathrm{d}\sigma_x (v) \, \nu(x).
    \end{equation*}
    By classical Riemannian geometry, $\mathrm{d}\nu_x$ is the Weingarten endomorphism at $x\in \partial\Omega$, a self-adjoint endomorphism of $T_x \partial\Omega$, so $\mathrm{d}\nu_x (v) \in T_x \partial\Omega$. Using the orthogonal decomposition $\R^d = T_x \partial\Omega \oplus \R \nu(x)$, the injectivity of $\mathrm{d}F_x$ amounts to the invertibility of $\mathrm{Id}_{T_x \partial\Omega} + \sigma(x)\, \mathrm{d}\nu_x$,
     which holds whenever {$|\sigma(x)| \, \norm{\mathrm{d}\nu_x}_{\mathrm{op}} < 1$, where $\norm{\cdot}_{\mathrm{op}}$ denotes the operator norm}. 
     By compactness of $\partial\Omega$, {up to shrinking $\varepsilon$ we may assume that $\varepsilon \, \max_{x\in \partial\Omega} \norm{\mathrm{d}\nu_x}_{\mathrm{op}} < 1$}, and then $F$ is an analytic immersion.
    By a result of Lee \cite[Theorem 6.24]{Lee2012}, {again up to shrinking $\varepsilon$,} the map $E : \partial \Omega \times (-\varepsilon, \varepsilon) \to \R^d$, $E(x, t) = x + t\, \nu(x)$, is a smooth embedding, hence injective. Thus $F(x) = E(x, \sigma(x))$ is injective, $F$ is an analytic embedding, and $\partial \Omega(\sigma) = F(\partial\Omega)$ is an analytic submanifold of $\R^d$.
    \medskip 

    We now prove the converse. Let $\sigma \in C^m(\partial\Omega)$ with $\norm{\sigma}_{C^m} < \varepsilon$ be such that $\partial \Omega(\sigma)$ is analytic.
    Fix $x_0 \in \partial \Omega$, set $y_0=x_0 + \sigma(x_0) \nu(x_0)$, and let $\nu_\sigma : \partial \Omega \to \R^d$ map each $x\in \partial \Omega$ to the unit exterior normal of $\partial \Omega(\sigma)$ at $x+\sigma(x)\nu(x)$.

    Taking an orthonormal basis of $T_{x_0} \partial \Omega$ and completing it with $\nu(x_0)$, the analyticity of $\partial \Omega$ provides an analytic real-valued function $f$ defined near $0$ in $\R^{d-1}$, with $f(0)=0$ and $\nabla f(0)=0$, and $A\in O_d(\R)$ {with $A e_d = \nu(x_0)$}, such that
    \begin{equation*}
        j : x' \mapsto x_0 + A(x', f(x'))
    \end{equation*}
    is an analytic parametrization of $\partial \Omega$ near $x_0$. Likewise, since $\partial \Omega(\sigma)$ is analytic, there exist an analytic function $g$ near $0$ in $\R^{d-1}$, with $g(0)=0$ and $\nabla g(0)=0$, and $B\in O_d(\R)$ {with $B e_d = \nu_\sigma(x_0)$}, such that
    \begin{equation*}
        k : y' \mapsto y_0 + B(y', g(y'))
    \end{equation*}
    is an analytic parametrization of $\partial \Omega(\sigma)$ near $y_0$. Define $\Psi(x',s) := j(x') + s\, \nu(j(x'))$ for $x'$ near $0$ in $\R^{d-1}$ and $s$ near $\sigma(x_0)$ in $\R$. This $\Psi$ is analytic and satisfies $\Psi(x', \sigma(j(x'))) \in \partial \Omega(\sigma)$. To apply the implicit function theorem, set
    \begin{equation*}
        F : (x',s) \mapsto g(\pi_{1,d-1}(B^{-1}(\Psi(x',s) - y_0))) - \pi_d(B^{-1}(\Psi(x',s) - y_0)),
    \end{equation*}
    where $\pi_{1,d-1} : \R^d \to \R^{d-1}$ and $\pi_d : \R^d \to \R$ are the canonical projections onto the first $d-1$ and the last coordinates, respectively. Then $F$ is analytic, as $g$ and $\Psi$ are, and $F(0, \sigma(x_0)) = 0$. Since $\nabla g(0)=0$,
    \begin{equation*}
        \partial_s F(0, \sigma(x_0)) =  -\pi_d(B^{-1}(\nu(x_0))).
    \end{equation*}
    It remains to show that $\pi_d(B^{-1}(\nu(x_0))) \neq 0$. To this end, a direct computation gives the expression
    \begin{equation*}
        {\nu_\sigma(x)=\frac{\nu(x)- (\mathrm{Id}+\sigma(x)\,\mathrm{d}\nu_x)^{-1}(\nabla_\tau \sigma(x))}{\sqrt{1+|(\mathrm{Id}+\sigma(x)\,\mathrm{d}\nu_x)^{-1}(\nabla_\tau \sigma(x))|^2}},}
    \end{equation*}
    {where $\nabla_\tau\sigma$ denotes the tangential gradient of $\sigma$. Indeed, the vector $(\mathrm{Id}+\sigma(x)\,\mathrm{d}\nu_x)^{-1}(\nabla_\tau\sigma(x))$ is tangent to $\partial\Omega$ and therefore orthogonal to $\nu(x)$. The displayed expression has unit length and is orthogonal to $\mathrm{d}F_x(v)$ for every $v\in T_x\partial\Omega$, using the self-adjointness of $\mathrm{Id}+\sigma(x)\,\mathrm{d}\nu_x$.}
    It follows from the above expression that $\nu_\sigma(x) = \nu(x) + O(\|\sigma\|_{C^1})$, so it points outward for $\varepsilon$ small as $\nu(x)$ points outward. Hence, for $\varepsilon>0$ small enough, $\scalp{\nu_\sigma(x_0)}{\nu(x_0)} \geq 1/2$, and since $B$ is orthogonal with $Be_d = \nu_\sigma(x_0)$,
    \begin{equation*}
        \pi_d(B^{-1}(\nu(x_0))) = \scalp{e_d}{B^{-1}(\nu(x_0))} = \scalp{B e_d}{\nu(x_0)}  = \scalp{\nu_\sigma(x_0)}{\nu(x_0)} \geq 1/2.
    \end{equation*}
    By the implicit function theorem, there is an analytic function $h$ defined near $0$ in $\R^{d-1}$ such that, near $(0, \sigma(x_0))$, $F(x',s)=0$ if and only if $s = h(x')$. In particular $\sigma(j(x')) = h(x')$ for $x'$ near $0$, so $\sigma$ is analytic near $x_0$. Since $x_0 \in \partial \Omega$ was arbitrary, $\sigma \in C^\omega(\partial\Omega)$.
\end{proof}

\subsection{Meagerness of smoother functions and conclusion}

Lemma~\ref{lem:param-ana} reduces the analyticity of the perturbed boundary $\partial\Omega(\sigma)$ to that of the scalar function $\sigma \in C^m(\partial\Omega)$. It therefore remains to show that analytic functions are meager in $C^m(\partial\Omega)$. We establish the following stronger statement: even $C^{m+1}$ is already meager in $C^m$.

\begin{lemma}\label{lem:Cm-meager}
    Let $M$ be a smooth compact manifold, then $C^{m+1}(M)$ is meager in $C^m(M)$.
\end{lemma}

\begin{proof}
    {Set $k := \dim M$. }Fix a smooth atlas of coordinate balls. By compactness, there are
    finitely many charts $\phi_a \colon U_a \to {\mathbb{R}^k}$ ($1 \le a \le N$),
    together with radii $r_a > 0$, such that $\overline{B(0,r_a)} \subset \phi_a(U_a)$
    and $M = \bigcup_{a=1}^N \phi_a^{-1}\!\big(B(0,r_a)\big)$. Write
    $B_a := \overline{B(0,r_a)}$ and, for $f \in C^m(M)$, let
    $f_a := f \circ \phi_a^{-1}$ denote its local representative. We equip $C^m(M)$
    with the norm
    \[
        \|f\|_{C^m} := \max_{1 \le a \le N}\ \max_{|\alpha| \le m}\ \sup_{B_a}
        \big|\partial^\alpha f_a\big|,
    \]
    which turns it into a Banach space; by compactness of $M$, any other such norm
    built from a finite atlas is equivalent. Fix $\beta \in (0,1)$ and, for $n \in \mathbb{N}^\ast$, set
    \[
        F_n := \Big\{ f \in C^m(M) \ :\ \forall a,\ \forall |\alpha| = m,\
        \forall u, v \in B_a,\quad
        \big|\partial^\alpha f_a(u) - \partial^\alpha f_a(v)\big| \le n\,|u - v| \Big\}.
    \]
    We show that each $F_n$ is closed with empty interior and that
    $C^{m+1}(M) \subset \bigcup_n F_n$.

    \medskip
    \emph{Closedness.} Let $f^{(j)} \to f$ in $C^m(M)$. For every $a$ and
    every $|\alpha| = m$ one has $\partial^\alpha f^{(j)}_a \to \partial^\alpha f_a$
    uniformly on $B_a$, so for fixed $u, v \in B_a$ the inequality
    $|\partial^\alpha f^{(j)}_a(u) - \partial^\alpha f^{(j)}_a(v)| \le n|u-v|$ passes
    to the limit. Hence $f \in F_n$, and $F_n$ is closed.

    \medskip
    \emph{Empty interior.} Since $C^{m+1}(M)$ is dense in $C^m(M)$ { (by mollification in charts and a partition of unity, see Hirsch\ \cite[Chapter~2]{Hirsch1976})}, it suffices to
    show that every $f \in C^{m+1}(M)$ is a limit, in $C^m(M)$, of functions outside
    $F_n$. Fix such an $f$ and $\varepsilon > 0$. Work in the chart $a = 1$ with
    coordinates ${w = (w_1, \dots, w_k)}$, and assume $0 \in \mathring{B_1}$. Let
    $\zeta(t) := |t|^{m+\beta}$, so that $\zeta \in C^m(\mathbb{R})$ with
    $\zeta^{(m)}(t) = c_{m,\beta}\,|t|^{\beta}\operatorname{sgn}(t)^{m}$ for a constant
    $c_{m,\beta} > 0$; in particular $\zeta^{(m)}$ is $\beta$-Hölder but not Lipschitz
    near $0$, as $|\zeta^{(m)}(t) - \zeta^{(m)}(0)|\,/\,|t| = c_{m,\beta}\,|t|^{\beta-1}
    \to +\infty$ as $t \to 0$. Pick a cutoff $\eta \in C_c^\infty(\mathring{B_1})$
    equal to $1$ near $0$, and set
    \[
        g := f + \delta\,\widetilde{h}, \qquad
        \widetilde{h} := \big(\eta\cdot(\zeta \circ w_1)\big) \circ \phi_1
        \ \text{ extended by } 0, \qquad
        \delta := \frac{\varepsilon}{2(\|\widetilde h\|_{C^m} + 1)},
    \]
    so that $\widetilde h \in C^m(M)$ and $\|g - f\|_{C^m} < \varepsilon$. Near $0$ we
    have $g_1 = f_1 + \delta\,\zeta(w_1)$. Taking $\alpha = (m, 0, \dots, 0)$ and points
    $u = (t,0,\dots,0)$, $v = 0$, the reverse triangle inequality gives
    \[
        \frac{\big|\partial^\alpha g_1(u) - \partial^\alpha g_1(v)\big|}{|t|}
        \ \ge\ \delta\,\frac{\big|\zeta^{(m)}(t) - \zeta^{(m)}(0)\big|}{|t|}
        - \frac{\big|\partial_1^m f_1(u) - \partial_1^m f_1(0)\big|}{|t|}
        \ \ge\ \delta\,c_{m,\beta}\,|t|^{\beta - 1} - L,
    \]
    where $L < \infty$ is a Lipschitz constant for $\partial_1^m f_1$ on $B_1$
    (finite since $f \in C^{m+1}$). The right-hand side tends to $+\infty$ as
    $t \to 0^+$, so  $g \notin F_n$.
    Thus every ball in $C^m(M)$ meets the complement of $F_n$, i.e. $F_n$ has empty
    interior.

    \medskip
    \emph{Inclusion $C^{m+1}(M) \subset \bigcup_n F_n$.} Let $f \in C^{m+1}(M)$. For
    each $a$ and $|\alpha| = m$, the function $\partial^\alpha f_a$ is $C^1$ on a
    neighborhood of $B_a$, hence
    \[
        L_a := \max_{|\alpha| = m}\ \sup_{B_a}\ \big|\nabla\,\partial^\alpha f_a\big|
        < \infty.
    \]
    As $B_a$ is convex, the mean value inequality yields, for all $u, v \in B_a$ and
    $|\alpha| = m$,
    \[
        \big|\partial^\alpha f_a(u) - \partial^\alpha f_a(v)\big| \le L_a\,|u - v|.
    \]
    With $L := \max_{1 \le a \le N} L_a < \infty$, any integer $n \ge L$ gives
    $f \in F_n$. Hence $C^{m+1}(M) \subset \bigcup_n F_n$.
\end{proof}

Combining the local chart of Lemma~\ref{lem:param-ana} with the meagerness of Lemma~\ref{lem:Cm-meager}, we can now conclude.

\begin{proof}[{Proof of Theorem~\ref{thm:meager}}]
    Let $\Omega \in \mathcal{A}$. By Lemma \ref{lem:param-ana}, there is an open neighborhood $V_\Omega$ of $\Omega$ in $\OO^m(D)$, homeomorphic to an open ball $B$ in $C^m(\partial\Omega)$, under which $\mathcal{A} \cap V_\Omega$ corresponds to $C^\omega(\partial\Omega) \cap B$. By Lemma \ref{lem:Cm-meager}, $C^\omega(\partial\Omega)\subset C^{m+1}(\partial\Omega)$ is meager in $C^m(\partial\Omega)$, hence $C^\omega(\partial\Omega) \cap B$ is meager in $B$, and thus $\mathcal{A} \cap V_\Omega$ is meager in $V_\Omega$. 

     Hence, since meagerness is a local property, see Kechris \cite[Theorem 8.29]{kechris1995classical}, $\mathcal{A}$ is meager in $G:= \bigcup_{\Omega \in \mathcal{A}} V_\Omega$. As $G$ is open in $\OO^m(D)$ and contains $\mathcal{A}$, we conclude that $\mathcal{A}$ is meager in $\OO^m(D)$.
\end{proof}

\section{Generic triviality of the isometry group}\label{sec:isometries}
Throughout this section $\Omega \in \OO^m(D)$ and
\begin{equation*}
    \mathrm{Iso}(\Omega) := \{ S \in \mathrm{Isom}(\R^d) \, : \, S(\Omega) = \Omega \}.
\end{equation*}
Every $S \in \mathrm{Isom}(\R^d)$ is a rigid motion $S(x) = Ax + c$ with
$A \in O_d(\R)$, so $S$ preserves the Lebesgue measure. Conversely, since
$\Omega$ is connected, any Riemannian isometry of $\Omega$ endowed with the
induced flat metric is also a rigid motion, so the two possible readings of
$\mathrm{Iso}(\Omega)$ agree. Indeed, let $\phi : \Omega \to \Omega$ be a Riemannian isometry. Then $\phi$ preserves the flat Levi-Civita connection, namely $\phi^* \nabla = \nabla$. Writing this equation in coordinates gives
$\partial_i \partial_j \phi^k = 0$,
and therefore, $\Omega$ being connected, $\phi$ is of the form $\phi(x)=Ax+b$, with $A = \mathrm{d}\phi \in O_d(\R)$ since $\phi$ is an isometry for the flat metric.
\subsection{The sign characters of a symmetry}

\begin{lemma}\label{lem:U_S}
Let $S \in \mathrm{Iso}(\Omega)$ and set $U_S f := f \circ S$. Then $U_S$ is a
unitary operator on $L^2(\Omega)$ which maps $H^1_0(\Omega)$
onto itself and commutes with the Dirichlet Laplacian, in the sense that
$U_S\bigl(H^2(\Omega) \cap H^1_0(\Omega)\bigr) = H^2(\Omega) \cap H^1_0(\Omega)$
and $\Delta (U_S f) = U_S (\Delta f)$ there. Moreover
$U_{S \circ T} = U_T \, U_S$ for all $S, T \in \mathrm{Iso}(\Omega)$.
\end{lemma}

\begin{proof}
Write $S(x) = Ax + c$ with $A \in O_d(\R)$. Since $S(\Omega) = \Omega$ and
$|\det A| = 1$, the change of variables $y = S(x)$ shows that $U_S$ is a linear
isometry of $L^2(\Omega)$, with inverse $U_{S^{-1}}$. As $S$ is affine,
$\nabla(f \circ S) = A^{\mathsf T} (\nabla f) \circ S$ and
$\Delta(f \circ S) = (\Delta f) \circ S$, because $A^{\mathsf T} A = \mathrm{Id}$.
Hence $U_S$ preserves $H^1(\Omega)$ and $H^2(\Omega)$, and it preserves the
vanishing of the trace since $S(\partial\Omega) = \partial\Omega$. Finally, for
$x \in \Omega$,
$(U_{S \circ T} f)(x) = f(S(T(x))) = (U_S f)(T(x)) = (U_T U_S f)(x)$.
\end{proof}

Let now $\Omega \in \mathcal{S}$, so that every $\lambda_n(\Omega)$ is simple and
the real $L^2$-normalized eigenfunction $\varphi_n$ is determined up to sign. By
Lemma~\ref{lem:U_S}, $U_S \varphi_n$ is again a real $L^2$-normalized
eigenfunction for $\lambda_n$, hence
\begin{equation}\label{eq:sign-character}
    \varphi_n \circ S = \varepsilon_n(S)\, \varphi_n
    \qquad \text{for a unique } \varepsilon_n(S) \in \{\pm 1\},
\end{equation}
and $\varepsilon_n(S)$ does not depend on the choice of sign for $\varphi_n$. From
$U_{S \circ T} = U_T U_S$ and \eqref{eq:sign-character} we get
$\varepsilon_n(S \circ T)\varphi_n = U_T(\varepsilon_n(S) \varphi_n)
= \varepsilon_n(S)\varepsilon_n(T)\varphi_n$, so that
\begin{equation*}
    \varepsilon_n : \mathrm{Iso}(\Omega) \longrightarrow \{\pm 1\}
\end{equation*}
is a group homomorphism, a \emph{sign character} of $\mathrm{Iso}(\Omega)$.

\subsection{Simplicity consequences on the isometry group}

The following proposition makes precise the intuitive statement that a simple
spectrum severely restricts the symmetries of $\Omega$. It shows that the isometry group must be a finite abelian group of exponent at most $2$. This was already known for compact connected Riemannian manifolds and was shown by Ikeda in \cite{Ikeda1985}. Here, using the affine structure of 
the isometry group, we also obtain a sharp bound $\# \mathrm{Iso}(\Omega) \leq 2^d$, which is new to our knowledge. See Example~\ref{ex:rectangle} for domains that saturate this bound.

\begin{proposition}\label{prop:iso-structure}
Let $\Omega \in \mathcal{S}$. Then the map
\begin{equation*}
    \varepsilon = (\varepsilon_n)_{n \geq 1} \, : \, \mathrm{Iso}(\Omega) \longrightarrow \{\pm 1\}^{\N}
\end{equation*}
is an injective group homomorphism. Consequently $\mathrm{Iso}(\Omega)$ is abelian
of exponent at most $2$, and there exists $0 \leq k \leq d$ with
$\mathrm{Iso}(\Omega) \cong (\Z/2\Z)^k$. In particular
$\# \mathrm{Iso}(\Omega) \leq 2^d$.
\end{proposition}

\begin{proof}
That $\varepsilon$ is a homomorphism is \eqref{eq:sign-character}. Suppose
$\varepsilon_n(S) = 1$ for every $n$. Then $U_S \varphi_n = \varphi_n$ for every
$n$, and since $(\varphi_n)_{n \geq 1}$ is an orthonormal basis of $L^2(\Omega)$,
$U_S = \mathrm{Id}$ on $L^2(\Omega)$. If $S \neq \mathrm{Id}$ then, $S$ being
affine and $\Omega$ open and nonempty, there is $x_0 \in \Omega$ with
$S(x_0) \neq x_0$. {Choose an open ball $B_0\Subset\Omega$ around $x_0$, small enough that $S(B_0)\cap B_0=\emptyset$, and take $f\in C_c(B_0)$ with $f\not\equiv0$. For $x \in B_0$ we have $S(x) \notin B_0$, so $U_Sf(x)=f( S(x))=0$: the function $U_S f$ vanishes identically on $B_0$, whereas $f$ does not. This contradicts $U_Sf=f$ in $L^2(\Omega)$.} Hence
$S = \mathrm{Id}$ and $\varepsilon$ is injective.

The target $\{\pm 1\}^{\N}$ is abelian of exponent $2$, so $\mathrm{Iso}(\Omega)$
is too. Let $b := |\Omega|^{-1}\int_\Omega x \, \df x$ be the barycentre of
$\Omega$. For $S(x) = Ax + c$ in $\mathrm{Iso}(\Omega)$, the change of variables
$y = S(x)$ gives $b = S(b)$, so after translating we may assume $b = 0$ and
$\mathrm{Iso}(\Omega) \subset O_d(\R)$. Every $A \in \mathrm{Iso}(\Omega)$ then
satisfies $A^2 = \mathrm{Id}$ and $A^{\mathsf T} = A^{-1} = A$, so $A$ is
symmetric. A commuting family of real symmetric matrices is simultaneously
orthogonally diagonalizable, so there is $P \in O_d(\R)$ with
$P^{-1}\mathrm{Iso}(\Omega) P \subset \{\mathrm{diag}(\pm 1, \dots, \pm 1)\}
\cong (\Z/2\Z)^d$. A subgroup of $(\Z/2\Z)^d$ is isomorphic to $(\Z/2\Z)^k$ for
some $k \leq d$.
\end{proof}

The deterministic statement behind Theorem~\ref{thm:isometry-generic} is now immediate.

\begin{proposition}\label{prop:iso-deterministic}
Let $\Omega \in \mathcal{S}$ be such that $\int_\Omega \varphi_n \neq 0$ for every
$n \geq 1$. Then $\mathrm{Iso}(\Omega) = \{\mathrm{Id}\}$.
\end{proposition}

\begin{proof}
Let $S \in \mathrm{Iso}(\Omega)$. Assume by contradiction that $S\not = \mathrm{Id}$, then there exists $n\geq 1$ such that
$\varepsilon_n(S) = -1$, and integrating \eqref{eq:sign-character} we get after change of variable that $\int_\Omega \varphi_n =0$, which is absurd.
\end{proof}

\begin{proof}[Proof of Theorem~\ref{thm:isometry-generic}]
By Theorem~\ref{thm:main} the set
$\mathcal{R} = \bigcap_{n \geq 1}(\mathcal{S}_n \setminus Z_n)$ is a dense
$G_\delta$ in $\OO^m(D)$, and every $\Omega \in \mathcal{R}$ lies in
$\mathcal{S}$ and has $\int_\Omega \varphi_n \neq 0$ for all $n$. By
Proposition~\ref{prop:iso-deterministic}, $\mathcal{R}$ is contained in
$\{\Omega : \mathrm{Iso}(\Omega) = \{\mathrm{Id}\}\}$, which therefore contains a
dense $G_\delta$ and is residual.
\end{proof}

To finish, we look at the family of rectangular boxes as an example that saturates the cardinality bound of Proposition~\ref{prop:iso-structure}. Even though these are not $C^m$-regular with $m \geq 3$, we include this example to exhibit the mechanism, as the results in this section do not specifically use the differentiability of the boundary.
\begin{example}\label{ex:rectangle}
Let $\Omega = \prod_{i=1}^{d}(0,a_i)$ be a rectangular box whose inverse squared
side lengths $a_1^{-2},\dots,a_d^{-2}$ are linearly independent over $\Q$. The
Dirichlet eigenpairs are
\begin{equation*}
    \lambda_n = \pi^2 \sum_{i=1}^{d} \frac{n_i^2}{a_i^2},
    \qquad
    \varphi_n(x) = \frac{2^{d/2}}{\sqrt{a_1\cdots a_d}}
    \prod_{i=1}^{d} \sin\Bigl(\frac{n_i \pi x_i}{a_i}\Bigr),
    \qquad n = (n_1,\dots,n_d) \in (\N^*)^d .
\end{equation*}
If $\lambda_n = \lambda_m$, then $\sum_{i=1}^{d}(n_i^2 - m_i^2)\,a_i^{-2} = 0$
with integer coefficients $n_i^2 - m_i^2 \in \Z$; the $\Q$-independence of the
$a_i^{-2}$ forces $n_i^2 = m_i^2$, hence $n_i = m_i$ for every $i$ (as
$n_i, m_i \geq 1$), so the spectrum is simple. (For $d = 2$ this is exactly the
condition $a_1^2/a_2^2 \notin \Q$.) The same hypothesis makes the $a_i$
pairwise distinct, so $\Omega$ carries no permutation symmetry and
\begin{equation*}
    \mathrm{Iso}(\Omega)
    = \bigl\{\, S_1^{\varepsilon_1}\cdots S_d^{\varepsilon_d}
    : \varepsilon \in \{0,1\}^d \,\bigr\}
    \cong (\Z/2\Z)^d,
    \qquad
    S_i(x) = (x_1,\dots,a_i - x_i,\dots,x_d),
\end{equation*}
which shows that the bound $k \leq d$ of Proposition~\ref{prop:iso-structure}
is attained. 
\end{example}

\section{\textcolor{black}{Applications to scalar parabolic control}}\label{sec:control}
{\color{black}
This section records the control-theoretic consequences of Theorem~\ref{thm:main}. Throughout, $D \subset \R^d$ is a bounded connected open set of class $C^m$ with $m \geq 3$ and $d \geq 2$, and $\Omega \in \OO^m(D)$. We fix $T > 0$ and consider the heat equation driven by a single scalar control acting uniformly on $\Omega$,
}
\begin{equation}\label{eq:heat-control}
\begin{cases}
\partial_t z = \Delta z + f(t)\,\mathbf{1}_\Omega & \text{in } \Omega, \\
z = 0 & \text{on } \partial \Omega, \\
z(0) = z_0 & \text{in } \Omega,
\end{cases}
\end{equation}
where $z_0 \in L^2(\Omega)$ is the initial data and $f \in L^2(0,T)$ is the control, which is \emph{homogeneous in space}: the same scalar $f(t)$ multiplies the constant profile $\mathbf{1}_\Omega$ over the whole domain. We say that \eqref{eq:heat-control} is \emph{approximately controllable} (in time $T$) if for every $z_0, z_1 \in L^2(\Omega)$ and every $\varepsilon > 0$ there exists a control $f \in L^2(0,T)$ such that the corresponding solution $z$ satisfies $\|z(T) - z_1\|_{L^2(\Omega)} \leq \varepsilon$.

\medskip 
For the reader less familiar with control theory, let us explain why this is
a meaningful question. The actuation in~\eqref{eq:heat-control} is extremely
poor: a single scalar function of time~$f$, whose spatial action is frozen to
the fixed profile~$\mathbf{1}_\Omega$. Inputs of this kind---in which time
and space are tensorized through a prescribed spatial profile---are known as
\emph{lumped controls}, following Tucsnak and Weiss \cite{TucsnakWeiss2009} and
Privat, Tr\'elat and Zuazua \cite{PrivatTrelatZuazua2017}. They
are considerably weaker than the distributed controls $\mathbf{1}_\omega
u(t,x)$ acting with full spatial freedom on a subdomain~$\omega$, for which
even null controllability holds in arbitrary time. {For such a degenerate,
one-dimensional input, the reachable set is governed by the Fourier coefficients of the fixed profile. Thus the appropriate robust target in the present paper is
\emph{approximate} controllability, reaching arbitrarily close to any $L^2$
target. Stronger notions, such as null controllability for rank-one inputs, require additional moment estimates and are subject to the obstructions already identified by Fattorini and Russell \cite{FattoriniRussell1971}.}

Whether approximate controllability holds for a \emph{generic} domain is a
natural question within a sustained program in the PDE control community of
studying domain-dependent spectral and controllability
properties~\cite{OrtegaZuazua2000,PrivatSigalotti2010,ChitourCoronGaravello2006}.
The question is concretely non-trivial: on a square or a ball,
symmetries force the mean of infinitely many eigenfunctions to vanish, so the
corresponding modes are invisible to the control and approximate
controllability fails outright. Theorem~\ref{thm:main} settles the question
in the affirmative: {the corresponding approximate controllability property holds
generically.}

\medskip

\subsection{The abstract control problem}

Let $\Omega \in \OO^m(D)$ and let $A_\Omega = \Delta$, with domain $D(A_\Omega) = H^2(\Omega) \cap H^1_0(\Omega)$, be the generator of the heat semigroup on $H := L^2(\Omega)$. Recall that $-A_\Omega$ is the Dirichlet Laplacian, with eigenvalues $\lambda_n > 0$ and $L^2$-normalized eigenfunctions $\varphi_n$ as in {Section~\ref{sec:setting}}. The spatial profile of the control is the \emph{bounded} control operator
\begin{equation*}
    B_\Omega := \mathbf{1}_\Omega \in L^2(\Omega),
\end{equation*}
and \eqref{eq:heat-control} rewrites as the abstract Cauchy problem
\begin{equation}\label{eq:abstract-heat}
\left\{
\begin{aligned}
\partial_t z &= A_\Omega z + f(t)\, B_\Omega \quad \forall t \in (0,T), \\
z(0) &= z_0.
\end{aligned}
\right.
\end{equation}
Since $A_\Omega$ generates an analytic semigroup, the approximate controllability of \eqref{eq:abstract-heat} is independent of $T > 0$ (see \cite{BadraTakahashi2014}), and we will simply speak of the approximate controllability of the pair $(A_\Omega, B_\Omega)$. Furthermore, when the spectrum of $A_\Omega$ is simple, keeping the notations of {Section~\ref{sec:setting}}, the Fattorini--Hautus criterion \cite{Fattorini1966,BadraTakahashi2014} characterizes it: the pair $(A_\Omega, B_\Omega)$ is approximately controllable if and only if $\scalp{B_\Omega}{\varphi_n(\Omega)}_{L^2(\Omega)} \neq 0$ for all $n \geq 1$. Here this pairing is nothing but the mean,
\begin{equation*}
    \scalp{B_\Omega}{\varphi_n(\Omega)}_{L^2(\Omega)} = \int_\Omega \varphi_n(\Omega),
\end{equation*}
so that the criterion is exactly the non-vanishing of all the eigenfunction means.

We can now answer our question in the affirmative, by recognizing it as a corollary of Theorem~\ref{thm:main}.

\begin{corollary}\label{cor:control}
    The set of $\Omega \in \OO^m(D)$ such that $(A_\Omega, B_\Omega)$ is approximately controllable is residual in $\OO^m(D)$.
\end{corollary}
\begin{proof}
    Let $\Omega \in \mathcal{R} = \cap_{n\geq 1} (\mathcal{S}_n \setminus Z_n)$, which is residual by the proof of Theorem~\ref{thm:main}. By definition the spectrum of $A_\Omega$ is simple and $\int_\Omega \varphi_n(\Omega) \neq 0$ for all $n \geq 1$, that is $\scalp{B_\Omega}{\varphi_n(\Omega)}_{L^2(\Omega)} \neq 0$ for all $n$. By the Fattorini--Hautus criterion, $(A_\Omega, B_\Omega)$ is approximately controllable, which concludes.
\end{proof}

Coming back to \eqref{eq:heat-control}, Corollary~\ref{cor:control} states that, for a generic domain,
 a single scalar control acting uniformly on $\Omega$ suffices to steer the heat equation approximately between any two $L^2$ states.
This implies that the scalar control acts nontrivially on
every eigenmode. For parabolic systems, this spectral nondegeneracy can be
upgraded to rapid stabilization by the $F$-equivalence method developed in
\cite{BoulardHayat2025}: for any prescribed decay rate, one constructs a
feedback that modifies the finitely many modes that decay too slowly, while
leaving the already sufficiently stable high-frequency modes unchanged. This
yields the following consequence.
\begin{corollary}[Generic rapid stabilization]
\label{cor:rapid-stabilization}
Let $D\subset\R^d$ be a bounded connected open set of class $C^m$, with
$m\geq 3$ and $d\geq 2$. For a residual set of domains
$\Omega\in\OO^m(D)$, the following property holds: for every prescribed
decay rate $\omega>0$, there exists a bounded linear feedback
$K_{\Omega,\omega}\in\mathcal L(L^2(\Omega),\R)$
such that the closed-loop system
\[
\begin{cases}
\partial_t z
=
\Delta z
+
K_{\Omega,\omega}(z)\,\mathbf 1_\Omega,
& \text{in }(0,+\infty)\times\Omega,\\
z=0,
& \text{on }(0,+\infty)\times\partial\Omega,\\
z(0)=z_0,
& \text{in }\Omega,
\end{cases}
\]
is well posed and satisfies
\[
\|z(t)\|_{L^2(\Omega)}
\leq
C_{\Omega,\omega}e^{-\omega t}
\|z_0\|_{L^2(\Omega)},
\qquad t\geq 0,
\]
for some constant $C_{\Omega,\omega}>0$ independent of $z_0$.
In particular, the Dirichlet heat equation is generically rapidly
stabilizable by a single spatially homogeneous scalar control.
\end{corollary}

\begin{proof}
By Theorem~\ref{thm:main}, generically the Dirichlet spectrum is simple and
$
\int_\Omega\varphi_n(\Omega)\neq 0$ for every $n\geq1$.
The scalar control operator $B_\Omega$
therefore satisfies the spectral admissibility condition required in
\cite[Theorem~3.2]{BoulardHayat2025} for every prescribed rate
$\omega>0$. The existence of the feedback and the stated decay estimate
then follow directly from
\cite[Proposition~3.4]{BoulardHayat2025}.
\end{proof}

\subsection{Homogeneous boundary control}

\begin{remark}\label{rem:boundary-control}
{The same argument applies, \emph{mutatis mutandis}, to a scalar control imposed uniformly on the boundary. We simply have to work with the state space $H^{-1}(\Omega)$ instead of $L^2(\Omega)$, in order to ensure the admissibility of the control operator.} Consider
\begin{equation}\label{eq:heat-boundary}
\begin{cases}
\partial_t z = \Delta z & \text{in } \Omega, \\
z = c(t) & \text{on } \partial \Omega, \\
z(0) = z_0 & \text{in } \Omega,
\end{cases}
\end{equation}
with $c \in L^2(0,T)$ the same scalar prescribed along all of $\partial\Omega$. Multiplying the first line by $\varphi \in C^\infty(\overline{\Omega})$ with $\varphi_{|\partial\Omega} \equiv 0$ and integrating by parts identifies the control operator $B_\Omega$, now defined by
\begin{equation*}
    \scalp{B_\Omega}{\varphi}_{D(A_\Omega)',D(A_\Omega)} := -\int_{\partial\Omega} \frac{\partial \varphi}{\partial \nu}.
\end{equation*}
This functional extends continuously to
$H^{\frac32+\varepsilon}(\Omega)\cap H_0^1(\Omega)$ for every
$\varepsilon>0$. Equivalently,
$B_\Omega\in
D ((-A_\Omega)^{3/4+\varepsilon})'$
where $A_\Omega$ denotes the Dirichlet Laplacian on $L^2(\Omega)$.
Viewing the system on the state space $H^{-1}(\Omega)$ shifts the
fractional scale by one half, so that $B_\Omega$ has degree of
unboundedness at most $1/4+\varepsilon$. In particular, its degree
of unboundedness is strictly smaller than $1/2$, hence
$B_\Omega$ is an admissible control operator (see, e.g.,
\cite[Chapter~5]{Trelat2024}), and the Fattorini--Hautus criterion
for unbounded control operators \cite{BadraTakahashi2014}
applies.

For an eigenfunction, Green's formula gives
\begin{equation*}
    \scalp{B_\Omega}{\varphi_n(\Omega)}_{D(A_\Omega)',D(A_\Omega)}
    = -\int_{\partial\Omega} \partial_\nu \varphi_n(\Omega)
    = \lambda_n \int_\Omega \varphi_n(\Omega),
\end{equation*}
and since $\lambda_n > 0$ the controllability condition is again $\int_\Omega \varphi_n(\Omega) \neq 0$ for all $n$. Consequently the conclusion of Corollary~\ref{cor:control} holds verbatim for \eqref{eq:heat-boundary}: a generic domain is approximately controllable, and {rapidly stabilizable}, by a spatially homogeneous boundary control as well.

\medskip

{It is precisely in this boundary-control setting that the contrast with Chitour, Coron and Garavello \cite{ChitourCoronGaravello2006} is sharpest. For the same scalar actuation imposed uniformly on $\partial\Omega$, they prove that for a generic domain in dimension $d\geq2$ the heat equation is \emph{not} steady-state controllable. This is a stronger, different property from approximate controllability. Thus the two statements are complementary rather than contradictory: steady-state controllability generically fails, whereas the weaker approximate controllability property generically holds. Together they delineate which controllability notions survive for a spatially homogeneous boundary control on a generic domain.}
\end{remark}

\section*{Acknowledgements}
The author would like to thank Noé Blassel for his advice on shape derivatives, Thomas Borsoni for his questions and Jean-Michel Coron for his time. This work was supported by the ANR-Tremplin StarPDE (ANR-24-ERCS-0010).

\small
\bibliographystyle{abbrv}
\bibliography{references}

\end{document}